\documentclass[a4paper,10pt]{smfart}

\usepackage[english, french]{babel}
\usepackage[utf8]{inputenc}
\usepackage[T1]{fontenc}

\usepackage[hidelinks]{hyperref}

\usepackage{amsmath}
\usepackage{amsthm}
\usepackage{amssymb}

\usepackage{enumitem}
\setlist{nosep}
\setlist[enumerate]{label=\emph{\alph*)}}

\usepackage{xspace}
\usepackage{version}

\usepackage[all,2cell,cmtip]{xy}

\input{smfref}

\SilentMatrices

\theoremstyle{plain}
\newtheorem{thm}{Théorème}[section]
\newtheorem{prop}[thm]{Proposition}
\newtheorem{lemme}[thm]{Lemme}
\newtheorem{coro}[thm]{Corollaire}
\theoremstyle{remark}
\newtheorem{rem}[thm]{Remarque}
\newtheorem{exs}[thm]{Exemples}
\theoremstyle{definition}
\newtheorem{defi}[thm]{Définition}
\newtheorem{paragr}[thm]{} 
\newtheorem*{notations}{Notations et terminologie}

\let\ndef\emph
\let\forlang\emph
\let\nbd\nobreakdash

\makeatletter

\def\xpoint{\futurelet\@let@token\@xpoint}
\def\@xpoint{%
  \ifx\@let@token.\else
    .%
  \fi
  \xspace}

\makeatother

\newcommand\opcit{\forlang{op. cit}\xpoint}

\newcommand\zbox[1]{\makebox[0pt][l]{#1}}
\newcommand\pbox[1]{\zbox{\quad#1}}

\newcommand\C{\mathcal{C}}
\newcommand\D{\mathcal{D}}
\newcommand\clC{\mathsf{C}}
\newcommand\W{\mathcal{W}}
\newcommand\M{\mathcal{M}}
\newcommand\N{\mathcal{N}}
\newcommand\Mono{\mathsf{Mono}}

\DeclareMathOperator{\Hom}{\mathsf{Hom}}
\DeclareMathOperator{\Homi}{\underline{\mathsf{Hom}}}
\newcommand{\Ho}{{\operatorname{\mathrm{Ho}}}}

\newcommand\trancheColax[3]{{#1}/\negmedspace/_{\negmedspace {\rm{c}}}^{#2}{#3}}

\newcommand\Wdoo{\W_\infty}
\newcommand\WDoo{\mathsf{W}_\infty}

\newcommand\pref[1]{\widehat{#1}}

\newcommand{\Sd}{Sd\kern 1pt}
\newcommand{\Ex}{Ex}
\newcommand{\id}[1]{1^{}_{#1}}

\let\bord\partial
\let\hookto\hookrightarrow
\newcommand\Cof{\mathsf{Cof}}
\newcommand\Fib{\mathsf{Fib}}

\newcommand\jcite{\cite{Chiche2LocFond}\xspace}

\newcommand\Cat{{\mathcal{C}\mspace{-2.mu}\it{at}}}
\newcommand\tCat{\texorpdfstring{\Cat}{Cat}}
\newcommand\dCat{{2\hbox{\protect\nbd-}\kern1pt\Cat}}
\newcommand\tdCat{\texorpdfstring{\dCat}{2-Cat}}

\newcommand\Top{{\mathcal{T}\mspace{-2.mu}\it{op}}}

\newcommand\dN{N_2}
\newcommand\NerfLaxNor{N_{\rm{l,n}}}

\let\carre\square

\newcommand\cocoin{\mathchoice{\mathord{\vrule height 6.29pt depth 0pt
width .34pt \vrule height 6.29pt depth -5.95pt width
5.95pt\,}}{\mathord{\vrule height 6.29pt depth 0pt width .34pt \vrule height
6.29pt depth -5.95pt width 5.95pt\,}}{\mathord{\hskip 1pt\vrule height 4.5pt
depth 0pt width .3pt \vrule height 4.5pt depth -4.2pt width
4pt\,}}{\mathord{\vrule height 2.9pt depth 0pt width .25pt \vrule height
2.9pt depth -2.65pt width 2.65pt\hskip .5pt}}}

\makeatother 

\title[Structures de catégories de modèles à la Thomason sur
$\dCat$]{Structures de catégorie de modèles à la Thomason sur la catégorie
des\\ $2$-catégories strictes}

\alttitle{Model category structures à la Thomason on $\dCat$}

\author{Dimitri Ara}
\address{Radboud Universiteit Nijmegen\\
Institute for Mathematics, Astrophysics and Particle Physics\\
Heyendaalseweg 135\\
6525 AJ Nijmegen\\
The Netherlands}
\email{d.ara@math.ru.nl}
\urladdr{http://www.math.ru.nl/\raise -3.3pt\vbox{\hbox{$\widetilde{ \ }\,$}}dara/}

\date{}

\begin{document}

\frontmatter

\begin{abstract}
Dans son article \emph{Théories homotopiques des 2-catégories}, Jonathan
Chiche étudie les théories homotopiques sur $\dCat$, la catégorie des
petites $2$\nbd-catégories strictes, données par des classes d'équivalences
faibles qu'il appelle localisateurs fondamentaux de $\dCat$. Ces
localisateurs fondamentaux de $\dCat$ sont une généralisation
$2$-catégorique de la notion de localisateur fondamental dégagée par
Grothendieck dans \emph{Pursuing stacks}.  Dans ce texte, nous déduisons des
résultats de Jonathan Chiche et de résultats que nous avons obtenus en
collaboration avec Georges Maltsiniotis l'existence, pour essentiellement
tout localisateur fondamental~$\W$ de $\dCat$, d'une structure de catégorie
de modèles à la \hbox{Thomason} sur~$\dCat$ dont les équivalences
faibles sont les éléments de $\W$. Nous démontrons que les structures de
catégorie de modèles ainsi obtenues modélisent exactement les localisations
de Bousfield à gauche combinatoires de la théorie de l'homotopie classique
des ensembles simpliciaux.
\end{abstract}

\begin{altabstract}In his \hbox{paper} \emph{Théories homotopiques des
2-catégories}, Jonathan Chiche studies homotopy theories on $\dCat$, the
category of small strict $2$-categories, given by classes of weak
equivalences which he calls basic localizers of $\dCat$. These basic
localizers of~$\dCat$ are a $2$-categorical generalization of the notion of
a basic localizer introduced by Grothendieck in \emph{Pursuing stacks}. In
this paper, we deduce from the results of Jonathan Chiche and results we
have obtained with Georges Maltsiniotis that for essentially every basic
localizer $\W$ of $\dCat$, there exists a model category structure à la
Thomason on $\dCat$ whose weak equivalences are given by $\W$. We show that
these model category structures model exactly combinatorial left Bousfield
localization of the classical homotopy theory of simplicial sets.
\end{altabstract}

\maketitle

\tableofcontents

\mainmatter

\section*{Introduction}

Ce texte a été initialement écrit comme un appendice à l'article
\emph{Théories homotopiques des $2$-catégories} \jcite de
Jonathan Chiche. Sur une suggestion du rapporteur, il a été promu en un
article indépendant. Ainsi, même si notre texte se veut auto-contenu, nous
encourageons le lecteur à lire \opcit, et notamment son introduction, avant
le présent article.

\medskip

Rappelons le contexte dans lequel se place \jcite.
La topologie algébrique moderne tend à remplacer les espaces topologiques
par les objets plus combinatoires que sont les ensembles simpliciaux. Dans
\emph{Pursuing stacks} \cite{GrothPS}, Grothendieck propose d'aller plus loin
et de fonder la théorie de l'homotopie sur la notion de petite catégorie. Il
s'agit en quelque sorte de remonter d'un cran supplémentaire dans la chaîne
de foncteurs
\[ \Cat \xrightarrow{N} \pref{\Delta} \xrightarrow{|\ |} \Top, \]
où $N$ est le foncteur nerf des petites catégories vers les ensembles
simpliciaux et $|\ |$~est le foncteur de réalisation topologique. Cela est
licite en vertu d'un résultat de Quillen : si on note $\W^1_\infty$ la
classe des foncteurs dont le nerf est une équivalence d'homotopie faible
simpliciale, alors le foncteur nerf induit une équivalence de catégories
\[
   \Ho(\Cat) \to \Ho(\pref{\Delta})
\]
entre la catégorie $\Cat$ localisée en $\W^1_\infty$ et la catégorie
homotopique usuelle des ensembles simpliciaux. Grothendieck étudie donc
$\Cat$ munie de la classe $\W^1_\infty$. Il se rend compte que les résultats
qu'il obtient ne dépendent que de quelques propriétés de la classe
$\W^1_\infty$. Il appelle \emph{localisateur fondamental} toute classe de
foncteurs qui vérifie ces propriétés et continue son étude de la théorie de
l'homotopie de $\Cat$ dans ce cadre axiomatique. Il conjecture que
$\W^1_\infty$ est le plus petit localisateur fondamental. Cette conjecture
est démontrée par Cisinski dans \cite{CisinskiLFM}. La théorie de
l'homotopie de Grothendieck est exposée dans \cite{Maltsi}.

Dans \jcite (et dans sa thèse \cite{ChicheThese}), Jonathan
Chiche pose les premières bases d'une théorie de l'homotopie à la
Grothendieck de $\dCat$, la catégorie des petites $2$-catégories strictes.
Notons $\W^2_\infty$ la classe des $2$-foncteurs envoyés sur une équivalence
d'homotopie faible simpliciale par n'importe quel foncteur nerf raisonnable,
disons le nerf géométrique $N_2 : \dCat \to \pref{\Delta}$ pour fixer les
idées. Jonathan Chiche montre dans~\jcite (le résultat apparaît
en fait déjà sous une forme moins générale dans \cite{ChicheThmA}) que le
foncteur~$N_2$ induit une équivalence de catégories
 \[ \Ho(\dCat) \to \Ho(\pref{\Delta}),  \]
où $\Ho(\dCat)$ désigne la catégorie $\dCat$ localisée en $\W^2_\infty$.
Il définit par ailleurs une notion de localisateur fondamental de $\dCat$,
analogue $2$\nbd-catégorique de la notion de localisateur fondamental de Grothendieck.
Il exhibe une bijection entre les localisateurs fondamentaux de $\Cat$ et de
$\dCat$ compatible à la localisation. Il utilise cette bijection et le résultat
de minimalité de Cisinski pour montrer que $\W^2_\infty$ est le localisateur
fondamental de $\dCat$ minimal.

\medbreak

Dans une direction complémentaire à l'approche de Grothendieck, Thomason a
démontré dans \cite{Thomason} l'existence d'une structure de catégorie de
modèles sur $\Cat$ dont les équivalences faibles sont les éléments de
$\W^1_\infty$. Il résulte du théorème de Quillen cité plus haut que cette catégorie
de modèles est équivalente, au sens de Quillen, avec la structure de catégorie
de modèles classique sur les ensembles simpliciaux.

La synthèse de ces travaux de Grothendieck et de Thomason a été effectué par
\hbox{Cisinski} dans son livre \cite{Cisinski}. Celui-ci démontre que pour tout
localisateur fondamental~$\W$ de $\Cat$ satisfaisant à une hypothèse
ensembliste anodine, il existe une structure de catégorie de modèles à la
Thomason sur $\Cat$ dont les équivalences faibles sont les éléments de $\W$. Il
démontre de plus que les structures de catégorie de modèles ainsi obtenues sur
$\Cat$ modélisent exactement les localisations de Bousfield à gauche
combinatoires de la structure de catégorie de modèles classique sur les
ensembles simpliciaux.

La généralisation $2$-catégorique du théorème de Thomason a été obtenue par
l'auteur de ce texte et Georges Maltsiniotis dans \cite{AraMaltsi}. Plus
précisément, nous y démontrons l'existence d'une structure de catégorie de
modèles à la Thomason sur $\dCat$ dont les équivalences faibles sont les
éléments de $\W^2_\infty$. De plus, nous déduisons d'un résultat de
Jonathan Chiche déjà cité que cette structure est équivalente, au sens de
Quillen, avec la structure de catégorie de modèles classique sur les
ensembles simpliciaux. Il est à noter que le texte antérieur \cite{WHPT} traite
également la question d'une généralisation $2$-catégorique du théorème de Thomason
mais qu'il contient de sérieuses erreurs (voir l'introduction de
\cite{AraMaltsi} pour plus de détails).

\medbreak

Le but du présent texte est de démontrer l'analogue $2$-catégorique du
théorème de Cisinski sur les structures à la Thomason, généralisant ainsi
les résultats de \cite{AraMaltsi} sur la structure à la Thomason
$2$-catégorique à un localisateur fondamental de $\dCat$ essentiellement
quelconque. Plus précisément, nous montrons l'existence, pour tout
localisateur fondamental $\W$ de $\dCat$ satisfaisant à une hypothèse
ensembliste anodine, d'une structure de catégorie de modèles à la Thomason
sur~$\dCat$ dont les équivalences faibles sont les éléments de $\W$. On
obtient ainsi une famille de structures de catégorie de modèles à la
\hbox{Thomason} sur~$\dCat$ modélisant exactement les localisations de
Bousfield à gauche combinatoires de la structure de catégorie de modèles
classique sur les ensembles simpliciaux.  Nous donnons par ailleurs des
conditions sur un localisateur fondamental de $\dCat$ pour que la structure
à la Thomason associée, qui est toujours propre à gauche, soit propre à
droite.

\medskip

Les ingrédients utilisés dans cet article sont de trois types. En plus des
résultats de~\cite{AraMaltsi}, et en particulier l'existence d'une structure
de catégorie de modèles à la Thomason $2$-catégorique pour $\W^2_\infty$,
les résultats présentés ici dépendent de manière cruciale de la minimalité du
localisateur fondamental $\W^2_\infty$ de $\dCat$, obtenue
dans~\jcite (théorème~6.37).  Cette minimalité résulte du résultat analogue
pour les localisateurs fondamentaux de $\Cat$, démontré par Cisinski dans
\cite{CisinskiLFM}, et d'une bijection entre les localisateurs fondamentaux
de~$\dCat$ et les localisateurs fondamentaux de $\Cat$ (théorème~6.33 de
\jcite), bijection qui joue également un rôle important dans ce texte.
Enfin, nos preuves dépendent de manière essentielle de plusieurs résultats
obtenus par Cisinski dans son livre \cite{Cisinski}, et en particulier de la
bijection entre les localisateurs fondamentaux de $\Cat$ et les
«~$\Delta$-localisateurs test » (théorème~4.2.15 de \opcit).

\begin{notations}
Nous nous écarterons peu des notations et du vocabulaire de \jcite. On
notera $\Cat$ la catégorie des petites catégories et $\dCat$ la catégorie
des petites $2$-catégories strictes et des $2$-foncteurs stricts. On
supprimera systématiquement l'adjectif « strict », les bicatégories ne
jouant aucun rôle dans ce texte, et les $2$-foncteurs lax ou oplax ne jouant
qu'un rôle caché. La catégorie des préfaisceaux sur une petite catégorie~$A$ sera
notée $\pref{A}$. On notera $\Delta$ la catégorie des simplexes et en
particulier $\pref{\Delta}$ la catégorie des ensembles simpliciaux. On
notera $N$ le foncteur nerf $\Cat \to \pref{\Delta}$ et $i_\Delta :
\pref{\Delta} \to \Cat$ le foncteur associant à un ensemble simplicial sa
catégorie des éléments. La catégorie des foncteurs d'une catégorie $\C$ vers
une catégorie~$\D$ sera notée~$\Homi(\C, \D)$. On notera $\Delta_1$ la
catégorie correspondant à l'ensemble ordonné~$\{0 \le 1 \}$.
On s'écartera légèrement des notations de \jcite en notant $\dN : \dCat \to
\pref{\Delta}$ le foncteur nerf géométrique qui y est noté~$\NerfLaxNor$.
Enfin, si $I$ est une classe de flèches d'une catégorie~$\C$, on notera
$l(I)$ (resp.~$r(I)$) la classe des flèches de $\C$ ayant la propriété de
relèvement à gauche (resp. à droite) par rapport à $I$.
\end{notations}

\section{Rappels sur les localisateurs fondamentaux}

Dans cette section, on rappelle brièvement la définition des localisateurs
fondamentaux, introduits par Grothendieck dans \cite{GrothPS}, et de leur
généralisation $2$-catégorique, introduite par Chiche dans \jcite. Nous
renvoyons le lecteur à ce dernier texte ou à la thèse \cite{ChicheThese} de
Chiche pour plus de détails et références sur les localisateurs
fondamentaux.

\begin{defi}
Soit $\W$ une classe de flèches d'une catégorie $\C$. On dit que $\W$ est
\ndef{faiblement saturée} si elle satisfait aux conditions suivantes :
\begin{enumerate}[label={\rm(FS\arabic*)}, labelindent=\parindent, leftmargin=*]
  \item les identités des objets de $\C$ sont dans $\W$ ;
  \item la classe $\W$ satisfait à la propriété du 2 sur 3 ;
  \item toute flèche $i$ de $\C$ admettant une rétraction $r$ telle que $ri$
    soit dans $\W$ est elle-même dans $\W$.
\end{enumerate}
\end{defi}

\begin{rem}
La condition de faible saturation est une forme faible de la notion de
catégorie homotopique au sens de Dwyer, Hirschhorn, Kan et Smith
\cite{HomotCat}. Plus précisément, si $(\C, \W)$ est une catégorie
homotopique au sens de \opcit, alors la classe $\W$ de flèches de $\C$ est
faiblement saturée.
\end{rem}

\begin{paragr}
Si $u : A \to B$ est un foncteur et $b$ est un objet de $B$, on notera $A/b$
la catégorie «~comma », parfois notée $u \downarrow b$, dont les objets sont
les couples $(a, f : u(a) \to b)$, où $a$ est un objet de~$A$ et $f$ une
flèche de $B$, et dont les flèches sont les morphismes de~$A$ faisant
commuter les triangles évidents. On vérifie immédiatement que si
    \[
    \xymatrix@C=1pc@R=1.5pc{
      A \ar[rr]^u \ar[dr] & & B \ar[dl] \\
      & C &}
    \]
est un triangle commutatif de $\Cat$, alors pour tout objet $c$ de $C$, le
foncteur $u$ induit un foncteur $u/c : A/c \to B/c$ donné sur les objets par
$(a, f) \mapsto (u(a), f)$.
\end{paragr}

\begin{defi}[Grothendieck]
Un \ndef{localisateur fondamental de $\Cat$} est une classe $\W$ de foncteurs
satisfaisant aux conditions suivantes :
\begin{enumerate}[label={\rm(LF\arabic*)}, labelindent=\parindent, leftmargin=*]
  \item la classe $\W$ de flèches de $\Cat$ est faiblement saturée ;
  \item pour toute petite catégorie $A$ admettant un objet final, l'unique
    foncteur $A \to e$, où $e$ est la catégorie finale, est dans $\W$ ;
  \item pour tout triangle commutatif
    \[
    \xymatrix@C=1pc@R=1.5pc{
      A \ar[rr]^u \ar[dr] & & B \ar[dl] \\
      & C &}
    \]
    dans $\Cat$, si pour tout objet $c$ de $C$ le foncteur $u/c$ appartient
    à $\W$, alors le foncteur $u$ appartient à $\W$.
\end{enumerate}
\end{defi}

\begin{exs}\label{exs:loc_fond_Cat}
L'exemple paradigmatique de localisateur fondamental de $\Cat$ est la classe
des foncteurs dont le nerf est une équivalence d'homotopie faible
simpliciale.

Plus généralement, si $\W$ est la classe des équivalences faibles d'une
localisation de Bousfield à gauche de la structure de catégorie de modèles
classique sur les ensembles simpliciaux, alors la classe des foncteurs dont
le nerf est dans $\W$ est un localisateur fondamental de $\Cat$. (Et par
le théorème \ref{thm:bij_Cisinski}, dû à Cisinski, on obtient ainsi tous les
localisateurs fondamentaux de $\Cat$, à des restrictions ensemblistes près.)
\end{exs}

Passons maintenant à la généralisation $2$-catégorique de la notion de
localisateur fondamental de~$\Cat$.

\begin{paragr}
Si $u : A \to B$ est un $2$-foncteur et $b$ est un objet de $B$, on notera
$A/b$ la catégorie «~comma » $2$-catégorique définie de la manière suivante
:
\begin{itemize}
  \item les objets sont les couples $(a, f : u(a) \to b)$, où $a$ est
    un objet de $A$ et $f$ une $1$-flèche de $B$ ;
  \item si $(a, f)$ et $(a', f')$ sont deux objets, les $1$-flèches de
    source $(a, f)$ et de but $(a', f')$ sont les couples $(g : a \to a',
    \alpha : f'u(g) \to f$), où $g$ est une $1$-flèche de $A$ et $\alpha$ une
    $2$-flèche de $B$ ;
  \item si $(g, \alpha)$ et $(g', \alpha')$ sont deux $1$-flèches
    de source $(a, f)$ et de but $(a', f')$, les $2$-flèches de $(g,
    \alpha)$ vers $(g', \alpha')$ sont les $2$-flèches $\beta : g \to g'$ de
    $A$ telles que
    \[ \alpha'\circ (f' \ast u(\beta)) = \alpha, \]
\end{itemize}
les compositions et identités étant définies de la manière évidente.
Cette catégorie est notée $\trancheColax{A}{u}{b}$ dans \jcite, l'indice
$c$, pour « colax », indiquant l'orientation des $2$-flèches de $B$ apparaissant
dans la définition des $1$-flèches. On renvoie à la section 3 de \opcit
pour plus de détails. On vérifie, comme dans le cas catégorique, que si
    \[
    \xymatrix@C=1pc@R=1.5pc{
      A \ar[rr]^u \ar[dr] & & B \ar[dl] \\
      & C &}
    \]
est un triangle commutatif de $\dCat$ et $c$ est un objet de $C$, alors le
$2$-foncteur $u$ induit un $2$-foncteur $u/c : A/c \to B/c$.

On dira, suivant \jcite, qu'un objet $z$ d'une $2$-catégorie $A$
\ndef{admet un objet final} si pour tout objet $a$ de $A$, la catégorie
$\Homi_A(a, z)$ des flèches de $a$ vers $z$ admet un objet final.
\end{paragr}

\begin{defi}[Chiche]
Un \ndef{localisateur fondamental de $\dCat$} est une classe~$\W$ de
$2$-foncteurs satisfaisant aux conditions suivantes :
\begin{enumerate}[label={\rm(LF$_2$\arabic*)}, labelindent=\parindent, leftmargin=*]
  \item la classe $\W$ de flèches de $\dCat$ est faiblement saturée ;
  \item pour toute petite $2$-catégorie admettant un objet admettant un objet
    final, l'unique foncteur $A \to e$, où $e$ est la $2$-catégorie
    finale, est dans $\W$ ;
  \item pour tout triangle commutatif
    \[
    \xymatrix@C=1pc@R=1pc{
      A \ar[rr]^u \ar[dr] & & B \ar[dl] \\
      & C &}
    \]
    dans $\dCat$, si pour tout objet $c$ de $C$ le $2$-foncteur $u/c$
    appartient à $\W$, alors le $2$-foncteur $u$ appartient à $\W$.
\end{enumerate}
\end{defi}

\begin{paragr}\label{paragr:def_N2}
Pour donner des exemples de localisateurs fondamentaux de $\dCat$, nous
aurons besoin d'un foncteur nerf $2$-catégorique. Dans ce texte, nous
privilégierons le nerf géométrique $N_2 : \dCat \to \pref{\Delta}$.
Rappelons brièvement sa définition. Si $C$ est une $2$\nbd-catégorie, les
$n$\nbd-simplexes de $N_2(C)$ sont donnés par les $2$-foncteurs
$\widetilde{\Delta_n} \to C$, où $\widetilde{\Delta_n}$ est la $2$-catégorie
définie de la manière suivante :
\begin{itemize}
  \item ses objets sont les entiers $0$, $1$, \dots, $n$ ;
  \item si $i$ et $j$ sont deux objets, la catégorie des flèches de $i$ vers
    $j$ est donnée par l'ensemble des sous-ensembles de $\{i, \dots, j\}$
    contenant $i$ et $j$, ordonné par l'ordre opposé à l'inclusion,
\end{itemize}
les compositions et identités étant définies de la manière évidente.
\end{paragr}

\begin{exs}
Les exemples~\ref{exs:loc_fond_Cat} de localisateurs fondamentaux de $\Cat$
se généralisent en des exemples de localisateurs fondamentaux de $\dCat$ en
remplaçant le nerf usuel par le nerf géométrique. De fait, en vertu du
théorème~\ref{thm:corr_Chiche}, dû à Chiche, les localisateurs fondamentaux
de $\Cat$ sont en bijection canonique avec les localisateurs fondamentaux de
$\dCat$.
\end{exs}

\begin{defi}\label{defi:W-equiv}
Si $\W$ est un localisateur fondamental de $\Cat$ ou de $\dCat$, on
appellera \ndef{$\W$-équivalences} ses éléments. 
\end{defi}

\section{Localisateurs et accessibilité}

Le but de ce texte est d'associer à tout localisateur fondamental $\W$ de
$\dCat$ «~accessible au sens de Cisinski~» une structure de catégorie de
modèles sur $\dCat$ dont les équivalences faibles sont les éléments de $\W$.
Commençons par définir cette notion d'accessibilité.

\begin{defi}
Si $S$ est une classe de foncteurs (resp.~de $2$-foncteurs), on appellera
\ndef{localisateur fondamental de $\Cat$ (resp. de~$\dCat$) engendré par
$S$} l'intersection de tous les localisateurs fondamentaux de $\Cat$ (resp.
de $\dCat$) contenant $S$. (On vérifie immédiatement qu'on obtient bien
ainsi un localisateur fondamental.) On dira qu'un localisateur fondamental
de $\Cat$ (resp. de $\dCat$) est \emph{accessible au sens de Cisinski} s'il
est engendré  par un \emph{ensemble}.
\end{defi}

Pour démontrer l'existence de la structure de catégorie de modèles
annoncée, nous utiliserons la notion intermédiaire de $A$-localisateur
(dans le cas $A = \Delta$).

\begin{defi}[Cisinski]
Soit $A$ une petite catégorie. Notons $\Mono$ la classe des monomorphismes
de la catégorie $\pref{A}$ des préfaisceaux sur $A$. Un
\ndef{$A$-localisateur} est une classe $\W$ de flèches de $\pref{A}$
satisfaisant aux conditions suivantes :
\begin{enumerate}[label={\rm(LC\arabic*)}, labelindent=\parindent, leftmargin=*]
  \item la classe $\W$ satisfait à la propriété du deux sur trois ;
  \item on a l'inclusion $r(\Mono) \subset \W$ ;
  \item la classe $\Mono \cap \W$ est stable par image directe et composition
    transfinie.
\end{enumerate}
Si $\W$ est un $A$-localisateur, on appellera \ndef{$\W$-équivalences} les
éléments de $\W$.
\end{defi}

\begin{defi}
Si $S$ est une classe de flèches de $\pref{A}$, on appellera
\ndef{$A$-localisateur engendré par~$S$} l'intersection de tous les
localisateurs contenant $S$. (On vérifie immédiatement qu'on obtient bien
ainsi un $A$-localisateur.) On dira qu'un $A$-localisateur est
\ndef{accessible au sens de Cisinski} s'il est engendré par un
\emph{ensemble}.
\end{defi}

\begin{thm}[Cisinski]\label{thm:loc_Cisinski}
Soient $A$ une petite catégorie et $\W$ un $A$-localisateur. Les conditions
suivantes sont équivalentes :
\begin{enumerate}
  \item il existe une structure de catégorie de modèles combinatoire sur
    $\pref{A}$ dont les équivalences faibles sont les éléments de $\W$ et
    dont les cofibrations sont les monomorphismes ;
  \item le localisateur $\W$ est accessible au sens de Cisinski.
\end{enumerate}
\end{thm}

\begin{proof}
C'est une partie du théorème 1.4.3 de \cite{Cisinski}. 
\end{proof}

Si $\W$ est un $A$-localisateur, on appellera la structure de catégorie de
modèles sur $\pref{A}$ donnée par le théorème précédent la \ndef{structure
de catégorie de modèles sur $\pref{A}$ associée à $\W$}.

\medbreak

On rappelle qu'on note $i_\Delta : \pref{\Delta} \to \Cat$ le foncteur qui
associe à tout ensemble simplicial sa catégorie des éléments et $N : \Cat
\to \pref{\Delta}$ le foncteur nerf.

\begin{thm}[Cisinski]\label{thm:bij_Cisinski}
Le couple de foncteurs
\[
i_\Delta : \pref{\Delta} \to \Cat,
\qquad
N : \Cat \to \pref{\Delta}
\]
induit une bijection
\[ \W \mapsto N^{-1}(\W), \qquad \W \mapsto i_\Delta^{-1}(\W) \]
entre la classe des $\Delta$-localisateurs contenant les équivalences
d'homotopie faibles simpliciales et la classe des localisateurs fondamentaux
de $\Cat$. De plus, cette bijection préserve l'accessibilité au sens de
Cisinski.
\end{thm}

\begin{proof}
En vertu du théorème 4.2.15 de \cite{Cisinski}, l'application 
$\W \mapsto i_\Delta^{-1}(\W)$ définit une bijection préservant
l'accessibilité au sens de Cisinski entre la classe des localisateurs
fondamentaux de $\Cat$ dits modelables par~$\Delta$ et la classe des
$\Delta$-localisateurs dits test (voir pour ces deux notions la définition
4.2.21 de \opcit). Il résulte de la proposition 1.5.13 de \cite{Maltsi} que
tout localisateur fondamental de $\Cat$ est modelable par~$\Delta$, et du
corollaire 2.1.21 et de la proposition 3.4.25 de \cite{Cisinski} que les
$\Delta$-localisateurs test sont exactement les $\Delta$-localisateurs
contenant les équivalences d'homotopie faibles simpliciales. Enfin, le fait
que le foncteur $N$ induit un inverse de cette bijection est conséquence de
la remarque 4.2.16 de \cite{Cisinski} et de l'exemple 1.7.18 de~\cite{Maltsi}.
\end{proof}

On rappelle qu'on note $N_2 : \dCat \to \pref{\Delta}$ le foncteur nerf
géométrique (voir le paragraphe \ref{paragr:def_N2}).

\begin{thm}[Chiche]\label{thm:corr_Chiche}
Le couple de foncteurs
\[
\iota : \Cat \to \dCat,
\qquad
i_\Delta\dN : \dCat \to \Cat
\]
où $\iota$ désigne l'inclusion canonique, induit une bijection
\[ \W \mapsto \dN^{-1}i_\Delta^{-1}(\W), \qquad \W \mapsto \iota^{-1}(\W) =
\W \cap \Cat \]
entre la classe des localisateurs fondamentaux de $\Cat$ et la classe des
localisateurs fondamentaux de $\dCat$. De plus, cette bijection préserve
l'accessibilité au sens de Cisinski.
\end{thm}

\begin{proof}
Voir le théorème~6.33 et les propositions~6.47 et 6.48 de \jcite.
(Rappelons que le foncteur qu'on note dans ce texte $N_2$ est noté
$\NerfLaxNor$ dans \opcit.)
\end{proof}

\begin{coro}
Le couple de foncteurs
\[
\iota\, i_\Delta : \pref{\Delta} \to \dCat,
\qquad
\dN : \dCat \to \pref{\Delta}
\]
induit une bijection
\[ \W \mapsto \dN^{-1}(\W), \qquad \W \mapsto i_\Delta^{-1}\iota^{-1}(\W) \]
entre la classe des $\Delta$-localisateurs contenant les équivalences
d'homotopie faibles simpliciales et la classe des localisateurs fondamentaux
de $\dCat$. De plus, cette bijection préserve l'accessibilité au sens de
Cisinski.
\end{coro}

\begin{proof}
Cela résulte immédiatement des deux théorèmes précédents une fois qu'on a
remarqué que le premier d'entre eux entraîne l'égalité 
\[ \dN^{-1} i_\Delta^{-1} N^{-1}(\W) = \dN^{-1}(\W) \]
pour tout $\Delta$-localisateur $\W$ contenant les équivalences d'homotopie
faibles simpliciales.
\end{proof}

\begin{paragr}\label{paragr:trijection}
Les deux théorèmes et le corollaire précédents fournissent une
«~trijection~», qu'on appellera \ndef{trijection de Chiche-Cisinski}, entre
les localisateurs fondamentaux de $\Cat$, les localisateurs fondamentaux de
$\dCat$ et les $\Delta$-localisateurs contenant les équivalences d'homotopie
faibles simpliciales. De plus, cette trijection préserve l'accessibilité au
sens de Cisinski.
\end{paragr}

Bien que ce ne soit pas strictement nécessaire pour obtenir les résultats
principaux de ce texte, nous allons consacrer la fin de cette section à
comparer la notion d'accessibilité au sens de Cisinski à une notion plus
classique d'accessibilité.

\begin{defi}\label{def:access_fl}
Une classe d'objets d'une catégorie accessible $\C$ est dite
\ndef{accessible} si le foncteur d'inclusion de la sous-catégorie pleine
correspondante dans $\C$ est accessible, c'est-à-dire s'il existe un
cardinal régulier $\kappa$ pour lequel ces deux catégories sont
$\kappa$-accessibles et le foncteur d'inclusion commute aux limites
inductives $\kappa$-filtrantes. Une classe de flèches d'une catégorie
accessible $\C$ est dite \ndef{accessible} si elle est accessible considérée
comme classe d'objets de la catégorie des flèches $\Homi(\Delta_1, \C)$ de $\C$.
\end{defi}

\begin{thm}[Smith]\label{thm:Smith}
Soient $\C$ une catégorie localement présentable, $\W$ une classe de flèches
de $\C$ et $I$ un \emph{ensemble} de flèches de $\C$. On note $\Cof$ la
classe $lr(I)$. Alors les conditions suivantes sont équivalentes :
\begin{enumerate}
  \item\label{item:SmithA} il existe une structure de catégorie de modèles combinatoire sur
    $\C$ dont les équivalences faibles sont les éléments de $\W$ et dont les
    cofibrations sont les éléments de $\Cof$ ;
  \item\label{item:SmithB} les conditions suivantes sont satisfaites :
\begin{enumerate}[label={\rm(S\arabic*)}, labelindent=0.5\parindent, leftmargin=*]
  \item\label{item:Smith1} la classe $\W$ satisfait à la propriété du deux sur trois ;
  \item\label{item:Smith2} on a l'inclusion $r(I) \subset \W$ ;
  \item\label{item:Smith3} la classe $\Cof \cap \W$ est stable par image
    directe et composition transfinie ;
  \item\label{item:Smith4} la classe de flèches $\W$ est accessible.
\end{enumerate}
\end{enumerate}
\end{thm}

\begin{proof}
Voir par exemple le corollaire A.2.6.6 et la proposition A.2.6.8 de
\cite{LurieHTT} (en tenant compte du fait qu'une classe de flèches accessible
est stable par rétractes). Pour l'implication~\hbox{$\ref{item:SmithB} \Rightarrow
\ref{item:SmithA}$}, voir également \cite{RosickyComb}.
\end{proof}

\begin{coro}\label{coro:acc_delta_loc}
Un $A$-localisateur est accessible au sens de Cisinski si et seulement s'il
est accessible en tant que classe de flèches de $\pref{A}$.
\end{coro}

\begin{proof}
Soit $\W$ un $A$-localisateur. Fixons $I$ un modèle cellulaire de $\pref{A}$
au sens de Cisinski, c'est-à-dire un ensemble $I$ tel que $lr(I)$ soit la
classe des monomorphismes de $\pref{A}$. Un tel ensemble existe toujours en
vertu par exemple de la proposition~1.2.27 de \cite{Cisinski}. Il résulte alors
du théorème de Smith appliqué à $\W$ et $I$ et du théorème~\ref{thm:loc_Cisinski} de Cisinski
appliqué à $\W$ que les trois conditions suivantes sont équivalentes :
\begin{enumerate}[label=\alph*)]
  \item la classe de flèches $\W$ est accessible ;
  \item il existe une structure de catégorie de modèles sur $\pref{A}$ dont
    les équivalences faibles sont les $\W$-équivalences et dont les
    cofibrations sont les monomorphismes ;
  \item le localisateur $\W$ est accessible au sens de Cisinski,
\end{enumerate}
ce qui achève la démonstration.
\end{proof}

\begin{prop}
La trijection de Chiche-Cisinski préserve l'accessibilité au sens des
classes de flèches (définition~\ref{def:access_fl}).
\end{prop}

\begin{proof}
Les catégories $\pref{\Delta}$, $\Cat$ et $\dCat$ étant accessibles, tout
adjoint à gauche ou à droite entre ces catégories est accessible (voir la
proposition 2.23 de \cite{AdamRos}). On en déduit que les foncteurs~$N$,
$i_\Delta$, $\iota$ et $\dN$ sont accessibles. Le résultat est alors
conséquence du fait que l'image réciproque d'une classe de flèches
accessible par un foncteur accessible est accessible (voir la
remarque 2.50 de \opcit).
\end{proof}

\begin{coro}
Un localisateur fondamental de $\Cat$ (resp. de $\dCat$) est accessible au
sens de Cisinski si et seulement s'il est accessible en tant que classe de
flèches de $\Cat$ (resp. de $\dCat$).
\end{coro}

\begin{proof}
La trijection de Chiche-Cisinski préservant l'accessibilité au sens de
Cisinski et l'accessibilité en tant que classe de flèches, le résultat est
conséquence immédiate du fait que ces deux notions coïncident pour les
$\Delta$-localisateurs (corollaire~\ref{coro:acc_delta_loc}).
\end{proof}

\begin{paragr}
Les deux notions d'accessibilité coïncidant, nous parlerons maintenant
simplement de $A$-localisateurs (resp. de localisateurs fondamentaux de
$\Cat$, resp. de localisateurs fondamentaux de $\dCat$) \ndef{accessibles}.
\end{paragr}

\begin{rem}
Il résulte de la proposition 1.4.28 de \cite{Cisinski} (resp. de la
proposition~2.4.12 de~\cite{Maltsi}, resp. de notre future
proposition~\ref{prop:loc_stable_lim}) que tout $A$-localisateur (resp.
tout localisateur fondamental de $\Cat$, resp. tout localisateur
fondamental de~$\dCat$) est stable par limite inductive suffisamment
filtrante. En vertu du théorème 6.17 de~\cite{AdamRos}, l'axiome de
grands cardinaux appelé «~principe de Vop\u enka~» implique donc que tout
$A$-localisateur (resp. tout localisateur fondamental de $\Cat$, resp. tout
localisateur fondamental de $\dCat$) est accessible.
\end{rem}

\section{La structure à la Thomason « classique » sur $\tdCat$}

\begin{paragr}
On notera $\Wdoo$ la classe des $2$-foncteurs (stricts) qui sont envoyés sur
des équivalences d'homotopie faibles simpliciales par le foncteur $N_2 :
\dCat \to \pref{\Delta}$. (Cette classe est notée $\W^2_\infty$ dans
l'introduction du présent article et dans \jcite.) On
appellera \ndef{$\Wdoo$\nbd-équivalences} ses éléments, conformément à la
terminologie introduite dans la définition~\ref{defi:W-equiv}.
\end{paragr}

\begin{paragr}\label{paragr:KanQuillen}
On appellera \ndef{structure de catégorie de modèles de Kan-Quillen} la
structure de catégorie de modèles sur $\pref{\Delta}$, introduite par
Quillen dans \cite{Quillen}, dont les équivalences faibles sont les équivalences
d'homotopie faibles et dont les cofibrations sont les monomorphismes. On
rappelle que cette structure de catégorie de modèles est combinatoire et
propre (voir par exemple le théorème~2.1.42 de~\cite{Cisinski}), et qu'un
ensemble de générateurs pour les cofibrations est donné par
\[ I = \{i_n : \bord{\Delta_n} \hookto \Delta_n \mid n \ge 0\}, \]
où $\bord{\Delta_n}$ désigne le bord du $n$-simplexe $\Delta_n$ dans
$\pref{\Delta}$ et $i_n : \bord{\Delta_n} \hookto \Delta_n$ l'inclusion
canonique.
\end{paragr}

\begin{paragr}\label{paragr:Sd_Ex}
On rappelle (voir \cite{KanCSSC}) qu'on a une adjonction
\[ \Sd : \pref{\Delta} \rightleftarrows \pref{\Delta} : \Ex, \]
où $\Sd$ est le foncteur de subdivision barycentrique et $\Ex$ est le
foncteur de Kan, et des transformations naturelles 
\[ \alpha : \Sd \to \id{\pref{\Delta}}, \quad \beta : \id{\pref{\Delta}} \to
\Ex, \]
transposées l'une de l'autre, qui sont des équivalences d'homotopie faibles
argument par argument.
\end{paragr}

\begin{paragr}
Enfin, on rappelle que le foncteur $N_2 : \dCat \to \pref{\Delta}$ admet un
adjoint à gauche, qu'on notera $c_2 : \pref{\Delta} \to \dCat$ (voir par
exemple le paragraphe 5.10 de \cite{AraMaltsi}).
\end{paragr}

\begin{defi}
Une \ndef{cofibration de Thomason de $\dCat$} est un $2$-foncteur élément de
la classe~$lr(c_2 \Sd^2(I))$.
\end{defi}

\begin{thm}[Ara-Maltsiniotis]\label{thm:AraMaltsi}
La catégorie $\dCat$ admet une structure de catégorie de modèles
combinatoire propre dont les équivalences faibles sont les
$\Wdoo$\nbd-équivalences et dont les cofibrations sont les cofibrations de
Thomason.
\end{thm}

\begin{proof}
C'est une partie du théorème 6.27 de \cite{AraMaltsi}.
\end{proof}

On appellera \ndef{structure de catégorie de modèles à la Thomason sur
$\dCat$} la structure donnée par le théorème précédent.

\begin{thm}[Ara-Chiche-Maltsiniotis]\label{thm:AraChicheMaltsi}
Le couple de foncteurs adjoints
\[ c_2\Sd^2 : \pref{\Delta} \rightleftarrows \dCat : \Ex^2\dN \]
est une équivalence de Quillen, où $\dCat$ est munie de la structure de
catégorie de modèles à la Thomason et $\pref{\Delta}$ de la structure de
catégorie de modèles de Kan-Quillen.
\end{thm}

\begin{proof}
C'est le corollaire 6.32 de \cite{AraMaltsi}.
\end{proof}

\begin{rem}
Le résultat précédent est partiellement attribué à Chiche car il dépend de
manière essentielle du théorème~7.9 de \cite{ChicheThmA}.
\end{rem}

\begin{coro}\label{coro:inclus_Wdoo}
On a les inclusions
\[
c_2 \Sd^2(\WDoo) \subset \Wdoo
\quad\text{et}\quad
\Ex^2\dN(\Wdoo) \subset \WDoo,
\]
où $\WDoo$ désigne la classe des équivalences d'homotopie faibles
simpliciales. De plus, les morphismes d'adjonction
\[
c_2 \Sd^2 \Ex^2 \dN \to \id{\dCat}
\quad\text{et}\quad
\id{\pref{\Delta}} \to \Ex^2\dN c_2\Sd^2
\]
sont respectivement une $\Wdoo$-équivalence naturelle et une équivalence
d'homotopie faible simpliciale naturelle.
\end{coro}

\begin{proof}
Puisque tous les objets de la structure de Kan-Quillen sont cofibrants, le
foncteur de Quillen à gauche $c_2\Sd^2$ respecte les équivalences faibles
(pour les structures de catégorie de modèles du théorème précédent).
Le foncteur $\dN$ respectant les équivalences faibles par définition, il en
est de même du foncteur $\Ex^2\dN$ en vertu de l'existence de l'équivalence
faible naturelle $\beta : \id{\pref{\Delta}} \to \Ex$ du
paragraphe~\ref{paragr:Sd_Ex}.
Puisque $(c_2\Sd^2, \Ex^2\dN)$ est une équivalence de Quillen donnée par
des foncteurs qui respectent les équivalences faibles, l'unité et la coünité
de cette adjonction sont des équivalences faibles naturelles, ce qu'il
fallait démontrer.
\end{proof}

\section{Structures à la Thomason et localisateurs fondamentaux de
$\tdCat$}

\begin{paragr}\label{paragr:W_Delta}
Si $\W$ est un localisateur fondamental de $\dCat$, on notera 
$\W_\Delta$ le $\Delta$\nbd-localisateur associé dans la trijection de
Chiche-Cisinski (voir le paragraphe~\ref{paragr:trijection}). Ce
$\Delta$-localisateur est caractérisé par le fait qu'il contient les
équivalences d'homotopie faibles simpliciales et par l'égalité
\[ \W = \dN^{-1}(\W_\Delta). \]
Il résulte de l'existence de l'équivalence faible naturelle $\beta :
\id{\pref{\Delta}} \to \Ex$ du paragraphe~\ref{paragr:Sd_Ex} qu'on a
également
\[ \W = \dN^{-1}(\Ex^2)^{-1}(\W_\Delta). \]
Par ailleurs, le théorème~6.37 de \jcite donne l'inclusion $\Wdoo \subset
\W$ qui jouera un rôle important dans ce qui suit.
\end{paragr}

Nous utiliserons dans cette section le lemme de transfert classique suivant
:

\begin{lemme}\label{lemme:transfert}
Soient $\M$ une catégorie de modèles à engendrement cofibrant 
(au sens de la définition 11.1.1 de \cite{Hirschhorn}) engendrée par~$I$ et
$J$, $\N$ une catégorie complète et cocomplète, et
\[ F : \M \rightleftarrows \N : G \]
un couple de foncteurs adjoints. Notons $\W$ et $\Fib$ les classes des
équivalences faibles et des fibrations de $\M$ respectivement. On suppose
les conditions suivantes satisfaites :
\begin{enumerate}
  \item $F(I)$ et $F(J)$ permettent l'argument du petit objet (au sens de la
    définition 10.5.15 de \cite{Hirschhorn}) ;
  \item on a l'inclusion $G(lr(F(J))) \subset \W$.
\end{enumerate}
Alors $F(I)$ et $F(J)$ engendrent une structure de catégorie de modèles sur
$\M$ dont les classes des équivalences faibles et des fibrations sont
données par $G^{-1}(\W)$ et $G^{-1}(\Fib)$ respectivement. En particulier,
pour cette structure de catégorie de modèles sur $\N$, l'adjonction $(F, G)$
est une adjonction de Quillen.
\end{lemme}

\begin{proof}
Voir par exemple le théorème 11.3.2 de \cite{Hirschhorn}, la description des
fibrations résultant des égalités $r(lr(F(J)) = r(F(J)) = G^{-1}(r(J)) =
G^{-1}(\Fib)$.
\end{proof}

On vérifie facilement qu'un foncteur de Quillen à gauche qui respecte les
équivalences faibles respecte également les carrés homotopiquement
cocartésiens et que, si un tel foncteur est de plus une équivalence de
Quillen à gauche, alors il reflète les carrés homotopiquement cocartésiens.
Nous aurons besoin de l'énoncé analogue pour les \emph{équivalences} de
Quillen à \emph{droite}, énoncé sans doute bien connu mais pour lequel nous
n'avons pas réussi à trouver de référence dans la littérature.

\begin{lemme}\label{lemme:carres_cocart}
Soit $F$ une équivalence de Quillen à droite entre deux catégories de
modèles. On suppose que $F$ préserve les équivalences faibles. Alors $F$
préserve et reflète les carrés homotopiquement cocartésiens.
\end{lemme}

\begin{proof}
Notons $\carre$ la catégorie $\Delta_1 \times \Delta_1$ (de sorte que si
$\C$ est une catégorie, $\Homi(\carre, \C)$ est la catégorie des carrés
commutatifs dans $\C$), $\cocoin$ la sous-catégorie pleine de $\carre$ contenant
tous les objets de~$\carre$ excepté $(1, 1)$, et $i : \cocoin \to \carre$
le foncteur d'inclusion canonique. On rappelle que si $\M$~est une catégorie
de modèles, un carré commutatif de $\M$ vu comme un objet $X$ de
$\Homi(\carre, \M)$ est homotopiquement cocartésien si et seulement si, pour
tout objet $Y$ de $\Homi(\carre, \M)$, l'application canonique
\[ \Hom_{\Homi(\carre, \M)}(X, Y) \to \Hom_{\Homi(\cocoin, \M)}(i^\ast X,
i^\ast Y) \]
induit une bijection
\[ \Hom_{\Ho(\Homi(\carre, \M))}(X, Y) \to \Hom_{\Ho(\Homi(\cocoin, \M))}(i^\ast X,
i^\ast Y), \]
où $\Ho$ désigne le passage à la catégorie homotopique pour les équivalences
faibles argument par argument.

Soit maintenant $F : \M \to \N$ une équivalence de Quillen à droite
préservant les équivalences faibles. Puisque le foncteur $F$ préserve les
équivalences faibles, il induit un carré commutatif
\[
\xymatrix{
\Ho(\Homi(\carre, \M)) \ar[r]^{i^\ast} \ar[d]_{F_\ast} & \Ho(\Homi(\cocoin,
\M)) \ar[d]^{F_\ast} \\
\Ho(\Homi(\carre, \N)) \ar[r]_{i^\ast} & \Ho(\Homi(\cocoin, \N)) \pbox{.}
}
\]
Puisque les catégories $\cocoin$ et $\carre$ sont des catégories de Reedy,
il résulte du fait que $F$ est une équivalence de Quillen et de la théorie
des structures de catégorie de modèles de Reedy (voir par exemple la
proposition 15.4.1 de \cite{Hirschhorn}) que les foncteurs verticaux du
carré commutatif ci-dessus sont des équivalences de catégories. Le résultat
suit immédiatement.
\end{proof}

Revenons à nos localisateurs.

\begin{lemme}\label{lemme:inclus_W}
Soit $\W$ un localisateur fondamental de $\dCat$. On a les inclusions
\[
c_2\Sd^2(\W_\Delta) \subset \W
\quad\text{et}\quad
\Ex^2\dN(\W) \subset \W_\Delta.
\]
\end{lemme}

\begin{proof}
La seconde inclusion résulte du paragraphe~\ref{paragr:W_Delta}. Montrons
la première. Puisque $\W_\Delta$ contient les équivalences d'homotopie
faibles simpliciales, le corollaire~\ref{coro:inclus_Wdoo} entraîne que 
$f$~est une $\W_\Delta$-équivalence si et seulement si $\Ex^2\dN
c_2\Sd^2(f)$ en est une, c'est-à-dire si et seulement si $c_2\Sd^2(f)$ est
une $\W$-équivalence, ce qui achève la démonstration.
\end{proof}

\begin{prop}\label{prop:loc_stable_lim}
Tout localisateur fondamental de $\dCat$ est stable par limite inductive
filtrante. 
\end{prop}

\begin{proof}
La proposition résulte de l'énoncé analogue pour les localisateurs
fondamentaux de $\Cat$ (voir la proposition 2.4.12 de \cite{Maltsi}), de la
correspondance donnée par le théorème~\ref{thm:corr_Chiche} entre
les localisateurs fondamentaux de $\Cat$ et ceux de $\dCat$, et du fait que
les foncteurs $\dN$ et~$i_\Delta$ commutent aux limites inductives
filtrantes (le premier en vertu par exemple de la proposition~5.13 de
\cite{AraMaltsi} et le second car il admet un adjoint à droite).
\end{proof}

\begin{thm}\label{thm:Thomason_W}
Soit $\W$ un localisateur fondamental de $\dCat$ accessible.
La catégorie $\dCat$ admet une structure de catégorie de modèles
combinatoire propre à gauche dont les équivalences faibles sont les
$\W$-équivalences, dont les cofibrations sont les cofibrations de Thomason
de $\dCat$ et dont les fibrations sont les $2$-foncteurs $u$ tels que~$\Ex^2
\dN(u)$ est une fibration de la structure de catégorie de modèles sur
$\pref{\Delta}$ associée au $\Delta$-localisateur~$\W_\Delta$.
\end{thm}

\begin{proof}
Nous allons appliquer le lemme~\ref{lemme:transfert} à l'adjonction
\[ c_2\Sd^2 : \pref{\Delta} \rightleftarrows \dCat : \Ex^2\dN, \]
où $\pref{\Delta}$ est munie de la structure de catégorie de modèles
associée au $\Delta$-localisateur $\W_\Delta$. Soit $J$ un ensemble
engendrant les cofibrations triviales de cette structure.
Puisque la catégorie $\dCat$ est localement présentable, il suffit de
vérifier qu'on a l'inclusion
\[ \dN\Ex^2(lr(c_2\Sd^2(J))) \subset \W_\Delta, \]
ou encore, en vertu du paragraphe~\ref{paragr:W_Delta}, qu'on a l'inclusion
\[ lr(c_2\Sd^2(J)) \subset \W. \]
Le lemme~\ref{lemme:inclus_W} entraîne que la classe $c_2\Sd^2(J)$ est
incluse dans $\W$ et le théorème~\ref{thm:AraChicheMaltsi} que cette même
classe est incluse dans la classe $\Cof$ des cofibrations de Thomason
de~$\dCat$. Pour conclure, en vertu de l'argument du petit objet, il suffit
donc de montrer que $\Cof \cap \W$ est stable par rétractes, composition
transfinie et image directe. La stabilité par rétractes est immédiate et
celle par composition transfinie résulte de la
proposition~\ref{prop:loc_stable_lim}. Montrons la stabilité par image
directe.

Considérons un carré cocartésien
\[
\xymatrix{
A \ar[d]_i \ar[r]^u & C \ar[d]^{i'} \\
B \ar[r]_v & D
}
\]
de $\dCat$ où $i$ est une cofibration de Thomason de $\dCat$. 
Par propreté à gauche de la structure à la Thomason sur $\dCat$ (voir le
théorème \ref{thm:AraMaltsi}), ce carré est homotopiquement cocartésien pour
cette même structure. En vertu du théorème~\ref{thm:AraChicheMaltsi} et de
son corollaire~\ref{coro:inclus_Wdoo}, le foncteur $\Ex^2 N_2$ est une
équivalence de Quillen à droite respectant les équivalences faibles (pour
les structures de catégorie de modèles du théorème invoqué). La
proposition~\ref{lemme:carres_cocart} entraîne donc que le carré
\[
\xymatrix@C=3.5pc{
\Ex^2\dN(A) \ar[d]_{\Ex^2\dN(i)} \ar[r]^{\Ex^2\dN(u)} & \Ex^2\dN(C)
\ar[d]^{\Ex^2\dN(i')} \\
\Ex^2\dN(B) \ar[r]_{\Ex^2\dN(v)} & \Ex^2\dN(D)
}
\]
est homotopiquement cocartésien pour la structure de Kan-Quillen. (On
pourrait se débarrasser des $\Ex^2$ en utilisant l'équivalence faible
naturelle $\beta$ du paragraphe~\ref{paragr:Sd_Ex}.) Il résulte du
fait que~$\W_\Delta$ contient les équivalences d'homotopie faibles (et que
les deux structures de catégorie de modèles sur $\pref{\Delta}$ en jeu ont
mêmes cofibrations) que ce carré est également homotopiquement cocartésien
pour la structure de catégorie de modèles associée à $\W_\Delta$. 

Si maintenant $i$ est de plus une $\W$-équivalence, alors $\Ex^2\dN(i)$ est
une $\W_\Delta$\nbd-équivalence et il en est donc de même de $\Ex^2\dN(i')$, ce
qui prouve que $i'$ est une $\W$-équivalence et achève de vérifier la
stabilité de $\Cof \cap \W$ par image directe.

La propreté à gauche s'obtient en remplaçant $i$ par $u$ et $i'$ par $v$
dans l'argument du paragraphe précédent.
\end{proof}

\begin{rem}
Il résulte de la preuve du théorème précédent que la structure de catégorie
de modèles obtenue est à engendrement cofibrant engendrée par $c_2\Sd^2(I)$
et $c_2\Sd^2(J)$, où $I$ est l'ensemble du
paragraphe~\ref{paragr:KanQuillen} et $J$ est un ensemble engendrant les
cofibrations triviales de la structure de catégorie de modèles sur
$\pref{\Delta}$ associée au $\Delta$-localisateur $\W_\Delta$.  
\end{rem}

On appellera la structure de catégorie de modèles sur $\dCat$ donnée par le
théorème précédent la \ndef{structure de catégorie de modèles à la Thomason
associée à $\W$}.

\begin{thm}\label{thm:eq_Quillen_dCat_W}
Soit $\W$ un localisateur fondamental de $\dCat$ accessible.
Le couple de foncteurs adjoints
\[ c_2\Sd^2 : \pref{\Delta} \rightleftarrows \dCat : \Ex^2\dN \]
est une équivalence de Quillen, où $\dCat$ est munie de la structure de
catégorie de modèles à la Thomason associée à $\W$ et $\pref{\Delta}$ de la
structure de catégorie de modèles associée au $\Delta$-locali\-sateur~$\W_\Delta$.
\end{thm}

\begin{proof}
Le foncteur $c_2\Sd^2$ préserve les cofibrations par définition et les
équivalences faibles par le lemme~\ref{lemme:inclus_W}. Le couple de 
foncteurs $(c_2 \Sd^2, \Ex^2 \dN)$ est donc une adjonction de Quillen (cela
résulte également de la preuve du théorème~\ref{thm:Thomason_W}). Puisque
le foncteur $\Ex^2\dN$ préserve également les équivalences faibles, pour
montrer que cette adjonction de Quillen est une équivalence de Quillen, il
suffit de vérifier que l'unité et la coünité de l'adjonction
sont des équivalences faibles naturelles. Cela résulte immédiatement
du corollaire~\ref{coro:inclus_Wdoo} et de la minimalité de~$\Wdoo$.
\end{proof}

Le degré de généralité naturel des arguments prouvant les théorèmes
\ref{thm:Thomason_W} et \ref{thm:eq_Quillen_dCat_W} est donné dans lemme que
nous allons maintenant énoncer. L'ordre bourbachique aurait voulu qu'on
commence par démontrer ce lemme et qu'on en déduise ces deux résultats
(modulo la trijection de Chiche-Cisinski) en
l'appliquant à l'équivalence de Quillen du
théorème~\ref{thm:AraChicheMaltsi} et à la structure de catégorie de modèles
associée au $\Delta$-localisateur~$\W_\Delta$. Nous y avons renoncé pour des
raisons d'exposition.

\begin{lemme}\label{lemme:transfert_loc}
Soient
\[ F : \M \rightleftarrows \N : G \]
une équivalence de Quillen et $\M'$ une catégorie de modèles à engendrement
cofibrant engendrée par $I'$ et $J'$ avec même catégorie sous-jacente que~$\M$,
mêmes cofibrations et $\W_\M \subset \W_{\M'}$, où $\W_\M$ et $\W_{M'}$
désignent les classes des équivalences faibles de $\M$ et~$\M'$ respectivement.
On suppose les conditions suivantes satisfaites :
\begin{enumerate}
  \item on a $G(\W_\N) \subset \W_\M$, où $\W_\N$ désigne la classe des
    équivalences faibles de $\N$ ;
  \item les sources et buts des flèches de $J'$ sont cofibrants dans $\M$ ;
  \item la catégorie de modèles $\N$ est propre à gauche ;
  \item la classe $G^{-1}(\W_{\M'})$ est stable par limite inductive filtrante.
\end{enumerate}
Alors $F(I')$ et $F(J')$ engendrent une structure de catégorie de modèles
propre à gauche~$\N'$ sur la catégorie sous-jacente à $\N$ dont les classes
des équivalences faibles et des fibrations sont données par
$G^{-1}(\W_{\M'})$ et $G^{-1}(\Fib_{\M'})$ respectivement, où $\Fib_{\M'}$
désigne la classe des fibrations de $\M'$.  De plus, l'adjonction $(F, G)$
induit une équivalence de Quillen
\[ F : \M' \rightleftarrows \N' : G. \]
\end{lemme}

\begin{proof}
La preuve est une adaptation immédiate des preuves des théorèmes
\ref{thm:Thomason_W} et \ref{thm:eq_Quillen_dCat_W}.
\end{proof}

\section{Équivalences de Quillen avec $\tCat$}

\begin{paragr}
Si $\W$ est un localisateur fondamental de $\Cat$, on notera $\W_\Delta$ le
$\Delta$-localisateur associé dans la bijection donnée par le théorème
\ref{thm:bij_Cisinski}. Ce $\Delta$-localisateur est caractérisé par le fait
qu'il contient les équivalences d'homotopie faibles simpliciales et par
l'égalité
\[ \W = N^{-1}(\W_\Delta). \]
Comme dans le cas $2$-catégorique, on a également
\[ \W = N^{-1}(\Ex^{2})^{-1}(\W_\Delta). \]
Notons que si $\W$ est un localisateur fondamental de $\dCat$, la trijection
de Chiche-Cisinski donne l'égalité $(\W \cap \Cat)_\Delta = \W_\Delta$.
\end{paragr}

\begin{defi}
Une \ndef{cofibration de Thomason de $\Cat$} est un élément de
la classe~$lr(c\, \Sd^2(I))$, où $c : \pref{\Delta} \to \Cat$ désigne
l'adjoint à gauche du foncteur nerf $N$.
\end{defi}

\begin{thm}[Cisinski]\label{thm:Thom_Cat}
Soit $\W$ un localisateur fondamental de $\Cat$ accessible. La
catégorie~$\Cat$ admet une structure de catégorie de modèles combinatoire
propre à gauche dont les équivalences faibles sont les $\W$-équivalences,
dont les cofibrations sont les cofibrations de Thomason de $\Cat$ et dont
les fibrations sont les foncteurs $u$ tels que $\Ex^2 N(u)$ est une
fibration de la structure de catégorie de modèles sur $\pref{\Delta}$
associée au $\Delta$-localisateur~$\W_\Delta$.
\end{thm}

\begin{proof}
Cela résulte de la preuve du théorème 5.2.15 de \cite{Cisinski}, la structure de
catégorie de modèles sur $\Cat$ en jeu y étant obtenue en appliquant le
lemme de transfert à l'adjonction \hbox{$c\,\Sd^2 : \pref{\Delta}
\rightleftarrows \Cat : \Ex^2N$}, où $\pref{\Delta}$ est munie de la
structure de catégorie de modèles associée au
$\Delta$-localisateur~$\W_\Delta$.
\end{proof}

\begin{rem}\label{rem:Thom_Cat}
Il résulte également de la preuve du théorème 5.2.15 de \cite{Cisinski} que
l'adjonction $c\,\Sd^2 : \pref{\Delta} \rightleftarrows \Cat : \Ex^2N$ est une
équivalence de Quillen, où $\pref{\Delta}$ est munie de la structure de
catégorie de modèles associée à $\W_\Delta$. 
\end{rem}

\begin{rem}
Le théorème~\ref{thm:Thom_Cat} et la remarque~\ref{rem:Thom_Cat} résultent
également du lemme~\ref{lemme:transfert_loc} appliqué à l'équivalence de
Quillen définie par Thomason dans \cite{Thomason} et à la structure de
catégorie de modèles associée au $\Delta$-localisateur $\W_\Delta$ (en
utilisant la bijection de Cisinski du théorème \ref{thm:bij_Cisinski}).
\end{rem}

On appellera la structure de catégorie de modèles sur $\Cat$ donnée par le
théorème précédent la \ndef{structure de catégorie de modèles à la Thomason
associée à $\W$}.

\begin{thm}\label{thm:eq_Cat_dCat}
Soit $\W$ un localisateur fondamental de $\dCat$ accessible. Alors
l'adjonction
\[ \tau : \dCat \rightleftarrows \Cat : \iota, \]
où $\tau$ désigne l'adjoint à gauche du foncteur $\iota$,
est une équivalence de Quillen, où $\dCat$ (resp. $\Cat$) est munie de la
structure de catégorie de modèles à la Thomason associée à $\W$ (resp. à $\W
\cap \Cat$).
\end{thm}

\begin{proof}
Les équivalences faibles et les fibrations de ces deux structures sont
précisément les morphismes s'envoyant, \forlang{via} les foncteurs $\Ex^2
\dN$ et $\Ex^2 N$ respectivement, sur des équivalences faibles et des
fibrations de la structure de catégorie de modèles associée
au $\Delta$\nbd-localisateur~$\W_\Delta$ (voir les théorèmes
\ref{thm:Thomason_W} et~\ref{thm:Thom_Cat} pour les fibrations).
Il résulte ainsi immédiatement de l'isomorphisme $\dN\iota \simeq N$ que le
foncteur~$\iota$ préserve les équivalences faibles et les fibrations, et
donc que le couple $(\tau, \iota)$ forme une adjonction de Quillen.

Puisque le foncteur $\iota$ préserve les équivalences faibles, pour
conclure, il suffit de voir que $\iota$~induit une équivalence sur les
catégories homotopiques. Cela résulte du
théorème~6.33 de \jcite.
\end{proof}

\begin{rem}
On pourrait également déduire le théorème précédent d'une version
« fonctorielle » du lemme~\ref{lemme:transfert_loc} qu'on appliquerait au
triangle d'équivalences de Quillen
\[
\xymatrix@C=1pc@R=1pc{
  & \Delta \ar[dl]_{c\,\Sd^2} \ar[dr]^{c_2\Sd^2} & \\
  \dCat \ar[rr]_\tau & & \Cat \pbox{,}
}
\]
où $\Delta$ (resp.~$\Cat$, resp.~$\dCat$) est munie de la structure de
catégorie de modèles de Kan-Quillen (resp. de Thomason \cite{Thomason},
resp.~du théorème \ref{thm:AraMaltsi}), et à la structure de catégorie de
modèles associée au $\Delta$-localisateur~$\W_\Delta$.
\end{rem}

\section{Propreté à droite}

\begin{defi}
Soit $\clC$ une classe de petites catégories (resp. de petites
$2$\nbd-catégories). On appellera \ndef{localisateur fondamental de $\Cat$}
(resp. \ndef{de $\dCat$}) \ndef{engendré par $\clC$} le localisateur
fondamental engendré par la classe de flèches $\{ C \to e \mid C \in \clC
\}$, où $e$ désigne la catégorie finale.
\end{defi}

\begin{thm}[Cisinski]\label{thm:Cis_propre}
Soit $\W$ un localisateur fondamental de $\Cat$ accessible. Les conditions
suivantes sont équivalentes : 
\begin{enumerate}
\item la structure de catégorie de modèles à la Thomason sur $\Cat$ associée
  à $\W$ est propre ;
\item\label{item:def_propre} la structure de catégorie de modèles sur
  $\pref{\Delta}$ associée à $\W_\Delta$ est propre ;
\item $\W$ est engendré par un \emph{ensemble} de catégories.
\end{enumerate}
\end{thm}

\begin{proof}
Dans \cite{Cisinski}, Cisinski définit une notion de localisateur
fondamental de $\Cat$ propre (définition 4.3.21). Il résulte du
théorème~4.3.24 de \opcit et de la proposition 1.5.13 de~\cite{Maltsi}
qu'un localisateur fondamental $\W$ de $\Cat$ est propre si et seulement
s'il satisfait à la condition~\ref{item:def_propre} ci-dessus. Les
équivalences avec les deux autres conditions résultent alors des
théorèmes~5.2.15 et~6.1.11 de \cite{Cisinski}.
\end{proof}

\begin{lemme}\label{lemme:eng_cat}
Soit $\W$ un localisateur fondamental de $\dCat$. Les conditions suivantes
sont équivalentes :
\begin{enumerate}
  \item $\W$ est engendré par une classe (resp. un ensemble) de petites $2$-catégories ;
  \item $\W$ est engendré par une classe (resp. un ensemble) de petites catégories ;
  \item le localisateur fondamental $\W \cap \Cat$ de $\Cat$ est engendré
    par une classe (resp. un ensemble) de petites catégories.
\end{enumerate}
\end{lemme}

\begin{proof}
Si $\W$ est engendré par une classe $\clC$ de petites $2$-catégories, alors,
en vertu de la proposition~6.47 de \jcite, le localisateur
fondamental $\W \cap \Cat$ de $\Cat$ est engendré par la classe de
foncteurs $\{ i_\Delta\dN(C) \to i_\Delta\dN(e) \mid C \in \clC \}$.
Puisque $i_\Delta\dN(e) \simeq \Delta$ admet un objet final, par définition
des localisateurs fondamentaux de~$\Cat$, le foncteur $i_\Delta\dN(e) \to e$
est dans $\W \cap \Cat$. Par deux sur trois, le localisateur fondamental $\W
\cap \Cat$ est donc engendré par la classe de petites catégories
\[ \{ i_\Delta\dN(C) \mid C \in \clC \}. \]
Par ailleurs, il résulte de la proposition~6.48 de \jcite que si le
localisateur fondamental~$\W \cap \Cat$ est engendré par une classe de
petites catégories, le localisateur fondamental $\W$ est également engendré
par cette classe de petites catégories, ce qui achève la démonstration.
\end{proof}

\begin{thm}
Soit $\W$ un localisateur fondamental de $\dCat$ accessible. Les conditions
suivantes sont équivalentes :
\begin{enumerate}
\item\label{item:2_propre} la structure de catégorie de modèles à la
  Thomason sur $\dCat$ associée à $\W$ est propre ;
\item\label{item:1_propre} la structure de catégorie de modèles à la Thomason sur $\Cat$ associée au
  localisateur fondamental $\W \cap \Cat$ de $\Cat$ est propre ;
\item\label{item:Delta_propre} la structure de catégorie de modèles sur
  $\pref{\Delta}$ associée au $\Delta$-localisateur $\W_\Delta$ est propre ;
\item $\W$ est engendré par un \emph{ensemble} de petites $2$-catégories ;
\item $\W$ est engendré par un \emph{ensemble} de petites catégories.
\end{enumerate}
\end{thm}

\begin{proof}
L'équivalence entre les quatre dernières conditions résulte du théorème 
\ref{thm:Cis_propre} et du lemme précédent. Les implications
$\ref{item:Delta_propre} \Rightarrow \ref{item:2_propre} \Rightarrow
\ref{item:1_propre}$ résultent du fait que les foncteurs
\[ \Cat \xrightarrow[\hphantom{\Ex^2\dN}]{\iota} \dCat \xrightarrow{\Ex^2\dN} \pref{\Delta} \]
sont des foncteurs de Quillen à droite (voir les théorèmes
\ref{thm:eq_Cat_dCat} et \ref{thm:eq_Quillen_dCat_W}) qui préservent et
reflètent les équivalences faibles.
\end{proof}

\bibliography{biblio}
\bibliographystyle{mysmfplain}

\end{document}